\documentclass[preprint,12pt,3p]{elsarticle}
\usepackage{graphicx}
\usepackage[utf8]{inputenc}
\usepackage[T1]{fontenc}
\usepackage{textcomp}
\usepackage{hyperref}
\hypersetup{colorlinks,linkcolor={blue},citecolor={blue},urlcolor={red}}

\usepackage{amsmath}
\usepackage{amssymb}
\usepackage{amsthm}
\newtheorem{theorem}{Theorem}[section]
\newtheorem{lemma}[theorem]{Lemma}

\newtheorem{proposition}{Proposition}[section]

\newtheorem{remark}[theorem]{Remark}

\numberwithin{equation}{section}

\usepackage{amssymb}

\def\RR{I\!\!R}
\def\NN{I\!\! N}

\usepackage{amsmath}

\begin{document}

\begin{frontmatter}

\title{Descret Solution for a Nonlinear Parabolic Equations with Diffusion Terms in Museilak-Spaces}

\author[label1]{A. Aberqi}
\address[label1]{Sidi Mohammed Ben Abdellah University, National School of Applied Sciences,
Laboratory LAMA, Fez, Morocco.}
\ead{aberqi$\_$ahmed@yahoo.fr}
\author[label2]{M. Elmassoudi}
\address[label2]{Sidi Mohammed Ben Abdellah University, Faculty of Sciences Dhar El Mahraz,
Laboratory LAMA, Department of Mathematics, B.P 1796 Atlas Fez, Morocco.}
\ead{elmassoudi09@gmail.com}
\author[label3]{M. Hammoumi}
\address[label3]{Sidi Mohammed Ben Abdellah University, Poly-disciplinary Faculty of Taza,
Laboratory LSI, P.O. Box 1223 Taza Gare, Morocco.}
\ead{hammoumi.mohamed09@gmail.com}

\begin{abstract}
\par In this paper, we study a  class of  nonlinear evolution equation with damping arising in fluid dynamics and rheology.
The nonlinear term is monotone and possesses a convex potential but exhibits non-standard growth. The appropriate functional framework for such equations is the modulary Museilak-Spaces. We prove the existence and uniqueness of a weak solution using an approximation approach by combining an internal approximation with the backward Euler scheme, also a priori error estimate for the temporal semi-discretization is given.
\end{abstract}


\end{frontmatter}

\section{Introduction}
\par In the present paper, We are concerned with the approximation of the initial boundary value problem for a non-linear parabolic equation that reads,
\begin{equation}\label{eq00000}
    \left\{\begin{array}{ll}
      \frac{\partial b(u)}{\partial t}-\mbox{div}(a(\nabla u)+K(u))=f,
    & \mbox{in}\, Q,\\
                u(x,t)=0 & \mbox{on}\, \partial \Omega\times (0,T),\\
                b(u)(t=0)=b(u_{0}) & \mbox{in}\,  \Omega.
     \end{array}%
     \right.
\end{equation}
Here, $\Omega$ is a bounded open set of $\RR^d$ $(d\geq 2)$, $T>0$, and $Q=\Omega \times (0,T)$. The stress $a :\Omega\times \RR^d \to \RR^d$ is assumed to have the potential $\varphi:\Omega\times \RR^d \to \RR^+$. Such an equation arises fluid dynamics and rheology (See \citep{Gwiazda2010Math.Methods.Appl}, \citep{M83} ).\\
Throughout this paper, we assume that:\\
 The field $ b: {\RR}\rightarrow {\RR}$ is a strictly increasing $\mathcal{C}^{1}(\RR)$-function, $b(0) = 0$, and there exists $b_{0}\in \RR$ such that
\begin{equation}\label{eqb}
\hspace{-0.5cm}0<b_{0}<b'(s)<b_{1}=2b_{0} \, \ \mbox{for all}\,\ s\in \RR, 
  \end{equation}
For any Musielak-function $\varphi$ (see definition below \ref{Mf}), the stress $a:\Omega\times\RR^{d} \rightarrow \RR^{d}$ is a  continuous function such that  for all $\xi,\xi^{*}\in \RR^{d}$, for a.e. $x\in \Omega$
\begin{equation}\label{eqa1}
|a(x,\xi)|\leq \overline{ \varphi}^{-1}\varphi(x,|\xi|),
\end{equation}
\begin{equation}\label{eqa2}
(a(x,\xi)-a(x,\xi^{\ast}))(\xi-\xi^{\ast})>0,
\end{equation}
\begin{equation}\label{eqa3}
a(x,\xi)\xi\geq\nu \varphi(x,|\xi|) \quad \mbox{for some} \ \nu>0,
\end{equation}
The diffusion terms $K:\RR\rightarrow\RR^{d}$ is a continuous function  such that
\begin{equation}\label{eqPhi1}
 |K(s)|\leq\nu_{0} \overline{\varphi}^{-1}\varphi(x, \frac{s}{\lambda}) \ \mbox{for all} \,\ s \ \mbox{in} \ \RR,\quad \mbox{for some} \ \nu_{0}>0, 
   \end{equation}
and
   \begin{equation}\label{eqf}
   f\in L^{1}(0,T;L^{2}(\Omega)).
   \end{equation}

In classical Sobolev spaces, starting with the paper \cite{Vcchio-Posteraro1996}, the authors proved an existence result of a weak solution for the non coercive problem (\ref{eq00000}) in the stationary case $b (u) = 0$ using the
symmetrization method. More later Di Nardo et al. \cite{Di.Nardo2011} have shown the existence
of renormalized solution for the parabolic version, more precisely in the linear case
$b(u) = u$, and the uniqueness for such solutions in the paper \cite{Blanchard-Murat-Redwane2001}, A. Aberqi et al. \cite{AB.BJ.MK.RH2013,AB.BJ.MK.RH2016}
have proved the existence of a renormalized solution for (\ref{eq00000}) with more general parabolic
terms $b(x, s)$.\\
In the Orlicz spaces we refer to \cite{Aharouch-Bennouna2010} where L. Aharouch, J. Bennouna have
proved the existence and uniqueness of entropy solutions in the framework of
Orlicz Sobolev spaces $W^{1}_{0}L_{\varphi}(\Omega)$ assuming the $\triangle_{2}$-condition on the Orlicz-function
$\varphi$. Recently, the uniqueness of renormalized solution of (\ref{eq00000}) in the general case has been proven by A. Aberqi et al. in \cite{Aberqi-Bennouna-Elmassoudi-Hammoumi2019uniq} and by  F. KH. Mukminov in \cite{MUKMINOV2016,MUKMINOV2017} for the Cauchy problem for anisotropic parabolic equation using Kruzhkovis method of doubling the variable.\\
Concerning the Musielak spaces, these are spaces that generalize the Orlicz spaces, the Lebesgue spaces with weight, and the Lebesgue spaces with variable exponent, we refer to \citep{M83}.
To our knowledge, articles dealing with this type of problem numerically, in these spaces, are rare. this prompt us to think about contributing to this study.\\
The difficulty encountered during the proof of the existence and uniqueness of discrete solution, is that the fact the Musielak spaces $L_{\varphi}(Q)$ is not isometrically isomorphic to the Musielak space $L_{\varphi}(0,T;L_{\varphi}(\Omega))$, the term $K$ does not satisfy the coercivity condition, and the nonlinearities are characterized by an Musielak-function $\varphi$, for which the $\Delta_{2}$-condition not imposed, and thus  the spaces $L_{\varphi}(Q)$ and $W^{1,x}_{0}L_{\varphi}(Q)$ are not necessarily reflexive.\\
In this paper, we consider the  weak formulation of (\ref{eq00000}) and propose a convergent
full discretization combining a piecewise constant finite element approximation
with the backward Euler scheme, we construct an approximation solution sequence for problem (\ref{eq00000})
and establish a priori estimation. Our study is done on the isotropic case and  generalizes \cite{Emmrich2013} and \cite{Ruf2017}  where the authors studied only the case $b(u)=u$ and $K=0$.\\

The outline of this paper is structured as follows: In Sect. 2, we introduce the necessary notation,
give a brief introduction to Musielak-Orlicz spaces.  The description of the numerical method we employ, the construction of the Galerkin scheme,
the proof of existence and uniqueness of the numerical solution, and the derivation of a priori estimates
for the fully discrete solution and the discrete time derivative follow in Sect. 3. Finally, in Sect. 4, we show
convergence towards and, thus, existence of an exact solution (weak solution of \eqref{eq00000}), as well as its uniqueness. An error estimate for the temporal semidiscretization is contained in the “Appendix”.
\section{\textbf{Preliminaries and auxiliary results}}
\subsection{Musielak function}\label{Mf}
Let $\Omega$ be an open subset of $\RR^{d}$ $(d\geq 2)$ and let $\varphi$ be a real-valued function defined in $\Omega\times \RR^{+}$. The function $\varphi$ is called a Musielak function if
\begin{itemize}
  \item $\varphi(x,\, ^{.})$ is an N-function for all $x\in \Omega$ (i.e. convex, non-decreasing, continuous, $\varphi(x,0)=0$, $\varphi(x,0)>0$ for  $t>0$, $\displaystyle\lim_{t\rightarrow 0}\frac{\varphi(x,t)}{t}=0$ and $\displaystyle\lim_{t\rightarrow \infty}\frac{\varphi(x,t)}{t}=\infty$).
  \item $\varphi(^{.},t)$ is a measurable function for all $t\geq 0$.
\end{itemize}
  We put  $\varphi_{x}(t)= \varphi(x,t)$ and we associate its non-negative reciprocal function $\varphi_{x}^{-1}$ with respect to $t$, that is,
$\varphi_{x}^{-1}(\varphi(x,t))=\varphi(x,\varphi_{x}^{-1}(t))=t.$\\
The Musielak function $\varphi$ is said to satisfy the $\Delta_{2}-$ condition if for some $C>0,$ and a non negative function $h,$ integrable in $\Omega,$ we have
\begin{equation}\label{DeltaCND} 
\varphi(x, 2 t) \leq C \varphi(x, t)+h(x)\, \ \mbox{for all}\,\  x \in \Omega \,\  \mbox{and all}\,\  t \geq 0.
\end{equation}
When (\ref{DeltaCND} ) holds only for $t \geq t_{0}>0,$ then $\varphi$ is said to satisfy the $\Delta_{2}-$ condition near infinity.\\
 Let $\varphi$ and $\gamma $ be two Musielak functions, we say that  $\varphi$ dominates $\gamma $ and we write $\gamma \prec \varphi$ near infinity (respectively, globally) if there exist two positive constants $c$ and $t_{0}$ such that for a.e. $x\in \Omega$,
$\gamma(x,t)\leq \varphi(x,ct)$ for all $t\geq t_{0}$ (respectively, for all  $t\geq 0$ ). We say that $\varphi$ and $\gamma $ are equivalents, and we write $\varphi \sim \gamma $ if $\varphi$ dominates $\gamma $ and $\gamma $ dominates $\varphi$. Finally, we say that $\gamma $  grows essentially less rapidly than $\varphi$ at $0$ (respectively, near infinity), and we write $\gamma << \varphi$, if for every positive constant $c$, we have
$\displaystyle\lim_{t\rightarrow 0}\sup_{x\in\Omega}\frac{\gamma(x,ct)}{\varphi(x,t)}=0$ (respectively, $\displaystyle\lim_{t\rightarrow \infty}\sup_{x\in\Omega}\frac{\gamma(x,ct)}{\varphi(x,t)}\big)=0).$
\begin{proposition}(\cite{EAB20})
Let $\gamma<< \varphi$ near infinity and for all $t>0$, $\displaystyle\sup_{x\in\Omega}\gamma(x,t)<\infty$, then
 for all $\epsilon >0$, there exists $C_{\epsilon}>0$ such that
\begin{equation}\label{eqcroit-moins-vite}
\gamma(x,t)\leq  \varphi(x, \epsilon t)+ C_{\epsilon}, \forall t>0.
\end{equation}
\end{proposition}

\subsection{Musielak space}
  Let  $\varphi$ be a Musielak function and a measurable function $u:\Omega \rightarrow \RR$, we define the functional $\displaystyle\varrho_{\varphi,\Omega}(u)=\int_{\Omega} \varphi(x,|u(x)|)dx$. 
 The set $K_{\varphi}(\Omega)=\{u:\Omega \rightarrow \RR \quad\mbox{measurable}:\quad\varrho_{\varphi,\Omega}(u)<\infty\}$ is called the Musielak class. The Musielak space $L_{\varphi}(\Omega)$ is the vector space generated by $K_{\varphi}(\Omega)$, that is, $L_{\varphi}(\Omega)$ is the smallest linear space containing the set $K_{\varphi}(\Omega)$. Equivalently,
$$L_{\varphi}(\Omega)=\{u:\Omega \rightarrow \RR \quad\mbox{measurable}:\quad\varrho_{\varphi,\Omega}(\frac{u}{\lambda})<\infty ,\quad\mbox{for some} \quad \lambda>0\}.$$
 On the other hand, we put $\displaystyle\overline{\varphi}(x,s)=\sup_{t\geq 0}(st-\varphi(x,s))$. $\overline{\varphi}$ is called the Musielak function complementary to $\varphi$ (or conjugate of $\varphi$) in the sense of Young with respect to $s$.
We say that a sequence of function $u_{n}\in L_{\varphi}(\Omega)$ is modular convergent to $u\in L_{\varphi}(\Omega)$ if there exists a constant $\lambda>0$ such that
$\displaystyle\lim_{n\rightarrow\infty}\varrho_{\varphi,\Omega}(\frac{u_{n}-u}{\lambda})=0.$ This implies convergence for $\sigma(\prod L_{\varphi},\prod L_{\overline{\varphi}})$ (see\cite{BV12}).\\
 In the space $L_{\varphi}(\Omega)$, we define the following two norms:
\begin{center}
$\|u\|_{\varphi}=\inf\big\{ \lambda>0: \displaystyle\int_{\Omega} \varphi(x,\frac{|u(x)|}{\lambda})dx\leq 1\big\},$\\
\end{center}
which is called the Luxemburg norm, and the so-called Musielak norm
\begin{center}
$ \displaystyle\||u|\|_{\varphi,\Omega}=\sup_{\|v\|_{\overline{\varphi}}\leq 1}\displaystyle\int_{\Omega}|u(x)v(x)|dx,$\\
\end{center}
where $\overline{\varphi}$ is the Musielak function complementary to $\varphi$. These two norms are equivalent \cite{M83}. $K_{\varphi}(\Omega)$ is a convex subset of $L_{\varphi}(\Omega)$. We define $E_{\varphi}(\Omega)$ as the subset of $L_{\varphi}(\Omega)$ of all measurable functions $u:\Omega \mapsto\RR$ such that $\displaystyle\int_{\Omega} \varphi(x,\frac{|u(x)|}{\lambda})dx<\infty$ for all $\lambda>0.$
It is a separable space and $(E_{\varphi}(\Omega))^{*}=L_{\varphi}(\Omega)$ (\cite{KDB2020}). We have $E_{\varphi}(\Omega)=K_{\varphi}(\Omega)$ if and only if $\varphi$ satisfies the $\Delta_{2}-$condition for the large values of $t$ or for all values of $t$, according to whether $\Omega$ has finite measure or not.\\
 For two complementary Musielak functions $\varphi$ and $\overline{\varphi}$, we have (see\cite{BV12})\\
  the Young inequality,  $\displaystyle \hspace{0.5cm} st\leq \varphi(x,s)+\overline{\varphi}(x,t)$ for all $s,t\geq 0$ , $x\in\Omega,$\\
  the H\"{o}lder inequality, $\displaystyle \big| \int_{\Omega}u(x)v(x) dx \big|\leq \|u\|_{\varphi,\Omega}\||v|\|_{\overline{\varphi},\Omega}$, for all $u\in L_{\varphi}(\Omega)$, $v\in L_{\overline{\varphi}}(\Omega).$\\
\indent We say that a sequence $u_{n}$ converges to $u$ for the modular convergence in  $W^{1}L_{\varphi}(\Omega)$ or in $W_{0}^{1}L_{\varphi}(\Omega)$  if, for some $\lambda>0$,
 $$\lim_{n\rightarrow \infty}\overline{\varrho}_{\varphi,\Omega}\big(\frac{u_{n}-u}{\lambda}\big)=0.$$
\indent Let us define the following spaces of distributions:
$$W^{-1}L_{\overline{\varphi}}(\Omega)=\big\{f\in \mathcal{D}^{'}(\Omega): f=\sum_{\alpha \leq1}(-1)^{\alpha} D^{\alpha}f_{\alpha}, \ \mbox{where} \ f_{\alpha}\in L_{\overline{\varphi}}(\Omega)\big\},$$
$$W^{-1}E_{\overline{\varphi}}(\Omega)=\big\{f\in \mathcal{D}^{'}(\Omega): f=\sum_{\alpha \leq1}(-1)^{\alpha} D^{\alpha}f_{\alpha}, \ \mbox{where} \ f_{\alpha}\in E_{\overline{\varphi}}(\Omega)\big\}.$$
\begin{lemma}\label{lemp1}(\cite{KDB2020}) (Approximation result)
\indent Let $\Omega$ be a bounded Lipschitz domain in $\RR^{d}$ and let $\varphi$ and $\overline{\varphi}$ be two complementary Musielak functions which satisfy the following conditions:
\begin{itemize}
  \item there exists a  constant $ c>0 $ such that $\displaystyle\inf_{x\in \Omega}\varphi(x,1)>c,$
  \item there exists two  constants $A, B>0 $ such that for all $x,y \in\Omega$ with $|x-y|\leq\frac{1}{2}$, we have
  $$\frac{\varphi(x,t)}{\varphi(y,t)}\leq A|t|^{\big(\frac{B}{\log(\frac{1}{|x-y|})} \big)}\quad \mbox{for all } \quad t\geq1,$$
  \item   $\displaystyle \int_{K} \varphi(x,\lambda)dx<\infty,$ for any constant $\lambda > 0$ and for every compact $K \subset\Omega.$
  \item
  $\mbox{there exists a constant}\quad C>0 \quad \mbox{such that} \quad\overline{\varphi}(y,t)\leq C  \quad \mbox{a.e.} \quad in \quad\Omega.$
\end{itemize}
\end{lemma}
Under these assumptions $\mathcal{D}(\Omega)$ is dense in $L_{\varphi}(\Omega)$  with respect to the modular topology, $\mathcal{D}(\Omega)$  is dense in $W_{0}^{1}L_{\varphi}(\Omega)$  for the modular convergence and $\mathcal{D}(\overline{\Omega})$ is dense in $W_{0}^{1}L_{\varphi}(\Omega)$ for the modular convergence. Consequently, the action of a distribution $S$  in $W^{-1}L_{\overline{\varphi}}(\Omega)$  on an element $u$ of $W_{0}^{1}L_{\varphi}(\Omega)$ is well defined. It will be denoted by $\langle S,u \rangle$.
\begin{remark} The second condition in Lemma \ref{lemp1} coincides with an
	alternative log-H\"{o}lder continuity condition for the variable exponent $p$, namely, there exists $ A>0 $ such that for $ x, y$ close enough and each $t \in \RR^{N} $
	$$|p(x)-p(y)|\leq \frac{A}{\log(\frac{1}{|x-y|})}.$$
\end{remark}	

\subsection{Inhomogeneous Musielak-Sobolev spaces :}
Let $\Omega$ an bounded open subset $\RR^{d}$ and let $Q=\Omega \times ]0,T[$ with some given $T>0$. Let $\varphi$\,be an Musielak-function, for each $\alpha \in \NN^{d}$,\,denote by $\nabla _{x}^{\alpha}$\, the distributional derivative on $Q$\, of order $\alpha$
with respect to the variable $x \in \NN^{d}$.\,The inhomogeneous Musielak-Sobolev spaces are defined as follows,
\begin{equation}\label{eq21}
 W^{1,x}L_{\varphi}(Q) =\{u \in L_{\varphi}(Q):\nabla _{x}^{\alpha}u \in L_{\varphi}(Q), \forall \alpha \in \NN^{d} , |\alpha| \leq 1 \}
\end{equation}
$$W^{1,x}E_{\varphi}(Q) =\{u \in E_{\varphi}(Q):\nabla _{x}^{\alpha}u \in E_{\varphi}(Q),  \forall \alpha \in \NN^{d} , |\alpha| \leq 1 \}$$
The last space is a subspace of the first one, and both are Banach spaces under the norm
$$\|u\|=\sum _{|\alpha|\leq m}\|\nabla_{x}^{\alpha}u\|_{\varphi,Q}.$$
We can easily show that they form a complementary system when $\Omega$\, satisfies the Lipschitz domain \cite{B.V.1}.
\,These spaces are considered as subspaces of the product space $\Pi L_{\varphi}(Q)$ which have as many copies as there is $\alpha$-order derivatives,$|\alpha| \leq 1$.
We shall also consider the weak topologies $\sigma (\Pi L_{\varphi},\Pi E_{\psi})$\, and $\sigma (\Pi L_{\varphi},\Pi L_{\psi})$.\,
If $u \in W^{1,x}L_{\varphi}(Q)$\, then the function $:t \mapsto u(t)=u(t,.)$\, is defined on $(0,T)$\, with values $W^{1}L_{\varphi}(\Omega)$.\,
.\,
Furthermore the following imbedding holds :$$ W^{1,x}E_{\varphi}(Q) \subset L^{1}(0,T,W^{1,x}E_{\varphi}(\Omega)). $$\,
\begin{lemma}\label{lempoincare}\cite{A.B.S.T.}
Under the assumptions of lemma \ref{lemp1}, and by assuming that
$\varphi(x, .)$ decreases with respect to one of coordinate of $x$, there exists a constant
$\lambda > 0$ which depends only on $\Omega$ such that
\begin{equation}\label{eq100}
    \int_{Q}\varphi(x,|u|)dxdt\leq \int_{Q}\varphi(x,\lambda|\nabla u|)dxdt
\end{equation}
\end{lemma}

We end this section with some useful lemmas
\begin{lemma}\label{lemfnf}
    If $(u_{n})\subset L^{1}(\Omega)$ with $u_{n}\rightarrow u$ a.e. in $\Omega$, $u_{n}, u\geq0$ a.e. in $\Omega$ and $\displaystyle\int_{\Omega}u_{n}dx\rightarrow \int_{\Omega}u dx$, then $u_{n}\rightarrow u$  in $L^{1}(\Omega)$.
\end{lemma}

\begin{lemma}\label{LemmaModulConv}\cite{J.G.2}
Let $u_{n}, u \in L_{\varphi}(\Omega)$. If $u_{n}\rightarrow u$ with respect to the modular convergence, then $u_{n}\rightarrow u$ for $\sigma(\Pi L_{\varphi}, \Pi L_{\overline{\varphi}})$.
\end{lemma}
\begin{lemma}\label{Lembourness}\cite{Emmrich2013}
Let $\{\xi_{l}\}\in \mathcal{L}_{\varphi}(Q)$ and there exists a positive constatnt $C$ such that\\ $\displaystyle\int_{Q}\varphi(x,|\xi_{l}|)dxdt\leq C$ for all $l\in \NN$. Then there exists $\xi\in \mathcal{L}_{\varphi}(Q)$ and a subsequence, denoted by $l'$, such that  $\xi_{l'}\rightharpoonup \xi$ weakly in $L^{1}(Q)$ and
\[\int_{Q}\varphi(x,|\xi|)dxdt\leq \liminf_{l'\rightarrow\infty}\int_{Q}\varphi(x,|\xi_{l'}|)dxdt.\]
\end{lemma}
Now, we denote by $\gamma_{0}w$ the trace of $w:\overline{\Omega}\rightarrow \RR$ such that $\gamma_{0}w=w$ on $\partial\Omega$, for smooth $w$.
\begin{lemma}\label{LemWWWW}\cite{Emmrich2013}
Let $\displaystyle w\in \mathcal{W}$, where
\[\mathcal{W}=\Big\{w\in W^{1,1}(0,T;L^{2}(\Omega)): \nabla w\in(\mathcal{L}_{\varphi}(Q))^{d},\gamma_{0}(w(.,t))=0\ \mbox{for almost}\ t\in(0,T)\Big\}\]
 For any $\epsilon>0$ there is then a smooth function $w_{\epsilon}$ witch vanishes in $\partial\Omega\times(0,T)$ such that
 \[\|w_{\epsilon}-w\|_{W^{1,1}(0,T;L^{2}(\Omega))}<\frac{\epsilon}{2},\]
 and  for all $\eta\in L_{\overline{\varphi}}(Q)$
 \[\displaystyle|\int_{Q}(\nabla (w_{\epsilon}-w)\eta dxdt|<\frac{\epsilon}{2}.\]
\end{lemma}
\textbf{Notation}:
Let $X$ be a Banach space, by $\mathcal{C} ([0, T ];X)$, we denote the
usual space of continuous functions $u : [0, T ] \rightarrow X$, and $\mathcal{C}_{w} ([0, T ];X)$ denotes the
space of demicontinuous functions (i.e., continuous with respect to the weak topology in X).\\
And let $T_k,$ denotes the truncation function at level $k>0,$ defined on $\RR$ by
 $\displaystyle T_k(r)=\max(-k,\min(k,r)).$
\section{Description of the numerical method}
In this section, we will describe the numerical method we employ, the construction of the Galerkin scheme,
the proof of existence and uniqueness of the numerical solution, and the derivation of a priori estimates
for the fully discrete solution and the discrete time derivative.
\paragraph{\textbf{A full discretization:}}
\begin{enumerate}
\item \textbf{For the spatial descritization:}
   We consider a generalized internal approximation $(V_{m})_{m\in \NN}$ of the space
\[V=\{v\in L^{2}(\Omega): \nabla v\in (E_{\varphi}(\Omega))^{d},\gamma_{0}v=0\}, \quad \|v\|_{V}=\|v\|_{2,\Omega}+\|\nabla v\|_{\varphi,\Omega},\] and the restriction operators $R_{m}:V\rightarrow V_{m}$ such that for any sequence $\{m_{l}\}_{l\in \NN}$ with $m_{l}\rightarrow \infty$ as $l\rightarrow \infty$ there holds
\begin{equation}\label{restriOperat-Rm}
    R_{m_{l}}v \rightarrow v \quad \mbox{in}\quad V \quad \mbox{as}\quad l\rightarrow +\infty\quad v\in V.
\end{equation}
Since $V$ is separable Banach space, there exists a Galerkin basis and an internal approximation scheme for $V$ (for more details, see \citep{Ciarlet},\cite{Emmrich2013}) and \citep{Lars} for such construction of restriction operator.\\

\item \textbf{For the temporal descritization:} For $N\in \NN$ ($N\geq 1)$, let $\tau=\frac{T}{N}$ and $t_{n}=n\tau$ ($n=0,1,..., N)$, according to Taylor-Young's formula
 \[b(u(x,t_{n}))-b(u(x,t_{n-1}))=\partial_{t}b(u(x,t_{n-1}))\tau+\tau\epsilon(\tau)\quad\mbox{such that}\quad\lim_{\tau\rightarrow0}\epsilon(\tau)=0,\]
  then we seek to find a $\{u^{n}\}_{n=1}^{N}\subset V_{m}$ such that for $n=1,2,...,N$
\begin{equation}\label{eqNumMethod}
    \int_{\Omega}\Big[(\frac{b(u^{n})-b(u^{n-1})}{\tau})v+a(\nabla u^{n})\nabla v+K(u^{n})\nabla v\Big]dx=\int_{\Omega}f(.,t_{n})vdx,\, \forall v\in V_{m}.
\end{equation}
Here, $u^{0}\in V_{m}$ denotes a suitable approximation of the initial datum $u_{0}\in L^{2}(\Omega)$.\\
\begin{lemma}
Fix $u^{0}\in V_{m}$ and assume that \eqref{eqa1}- \eqref{eqPhi1} hold true, then there exists a weak solution of
\[\int_{\Omega}\Big[(\frac{b(u)-b(u^{n-1})}{\tau})v+a(\nabla u)\nabla v+K(u)\nabla v\Big]dx=\int_{\Omega}f(.,t_{n})vdx\quad \forall v\in V_{m}.
\] 
and if $K$ satisfies the condition
\begin{equation}\label{eqPhi2}
        |K(s)-K(s')|\leq \nu_{1} |s-s'| \quad \mbox{for all} \quad   s, s' \in \RR, \, \mbox{and some} \, \nu_{1}>0,
   \end{equation}
   then the solution is unique.
 \end{lemma}  
Indeed, for the existence of the weak solution to \eqref{eqNumMethod}, we refer to (\cite{Aberqi-Bennouna-Elmassoudi-Hammoumi2019uniq}, \citep{A19}).\\
For the uniqueness, let $v$ and $w$ are two solutions of \eqref{eqNumMethod} and taking $\frac{1}{k}T_{k}(v-w)$ as a test function,
thus
\[\frac{1}{ k}\int_{\Omega}  (\frac{(b(v)-b(w)}{\tau})T_{k}(v-w)dx =-\frac{1}{k}\int_{\Omega}(a(\nabla v)-a(\nabla w))\nabla T_{k}(v-w)dx\] \[\hspace{5cm}-\frac{1}{k}\int_{\Omega}(K(v)-K(w))\nabla T_{k}(v-w)dx.\]
We have
 \[\displaystyle\lim_{k\rightarrow 0}\frac{1}{ k}\int_{\Omega}   (\frac{b(v)-b(w)}{\tau})T_{k}(v-w)dx=\int_{\Omega}(\frac{b(v)-b(w)}{\tau}) sgn(v-w)dx\]
 \[\hspace{3cm}\geq \frac{b_{0}}{\tau}\int_{\Omega} |v-w|dx.\]
By \eqref{eqa2} we have $\displaystyle-\frac{1}{k}\int_{\Omega}(a(\nabla v)-a(\nabla w))\nabla T_{k}(v-w)dx\leq0$\\
Using \eqref{eqPhi2} we get
\[\frac{1}{k}\int_{\Omega}(K(v)-K(w))\nabla T_{k}(v-w)dx\leq \frac{1}{k}\int_{|v-w|\leq k}\nu_{1}|v-w||\nabla T_{k}(v-w)|dx\]
\[\hspace{3.7cm}\leq \int_{|v-w|\leq k}\nu_{1}|\nabla(v-w)|dx.\]
Since $|\nabla(v-w)|\in L^{1}(\Omega)$
then $\displaystyle\lim_{k\rightarrow 0}\frac{1}{k}\int_{\Omega}(K(v)-K(w))\nabla T_{k}(v-w)dx=0$.\\
Thus
\[\frac{b_{0}}{\tau}\int_{\Omega} |v-w|dx\leq0.\]
This implies that $v=w$.\\
We conclude that for any $u^{0}\in V_{m}$ and $f\in L^{1}(0,T;L^{2}(\Omega))$, there exists a unique solution $\displaystyle\{u^{n}\}_{n=1}^{N}\subset V_{m}$ to \eqref{eqNumMethod}.
\item \textbf{A priori estimates for the descret solution:}\\
In this section, we will derive some a priori estimates, which will play a crucial role in achieving the convergence of the numerical solution.
\begin{theorem}\label{Uniformbounded-discret-solution}
Let $u^{0}\in V_{m}$ and $f\in L^{1}([0,T];L^{2}(\Omega))$. Let $\displaystyle\{u^{n}\}_{n=1}^{N}\subset V_{m}$ be the solution to \eqref{eqNumMethod} and let $\tau\leq\tau_{0}<1$. Then there exist  a positive constants $c_{1}$, $c_{2}$, $c_{3}$ and $c_{4}$ depending on $p_{0}$, $p_{1}$, $\nu$, $T$ and $\tau_{0}$ such that for all $n=1,2,...,N$,
$$\|b(u^{n})\|_{2,\Omega}^{2}+c_{1}\sum_{j=1}^{n}\|b(u^{j})-b(u^{j-1})\|_{2,\Omega}^{2}+c_{2}\sum_{j=1}^{n}\int_{\Omega}\varphi(x,|\nabla u^{j}|)dx$$
\begin{equation}\label{eqUnf-bund-Sol}
\leq c_{3}\|f\|^{2}_{L^{1}([0,T];L^{2}(\Omega))}+c_{1}\|b(u^{0})\|_{2,\Omega}^{2},
\end{equation}
and
\begin{equation}\label{eqnormUn}
\sum_{j=1}^{n}\|u^{n}-u^{n-1}\|_{2,\Omega}\leq c_{4}.
\end{equation}
\end{theorem}
\begin{proof}
Taking $v=b(u^{n})$ in \eqref{eqNumMethod} for the discrete time derivative, we have
$$\|b(u^{n})\|_{2,\Omega}^{2}-\|b(u^{n-1})\|_{2,\Omega}^{2}+\|b(u^{n})-b(u^{n-1})\|_{2,\Omega}^{2}$$
$$+2\tau\int_{\Omega}[b'(u^{n})a( \nabla u^{n})\nabla u^{n}+b'(u^{n})K(\nabla u^{n})\nabla u^{n}]dx\leq 2\tau\int_{\Omega}|f_{n}(x)||b(u^{n})|dx$$
Using the assumptions \eqref{eq100}, \eqref{eqa3}, \eqref{eqPhi1}  and using the Young inequality, we obtain
$$\|b(u^{n})\|_{2,\Omega}^{2}-\|b(u^{n-1})\|_{2,\Omega}^{2}+\|b(u^{n})-b(u^{n-1})\|_{2,\Omega}^{2}+2\nu\tau b_{0}\int_{\Omega}\varphi(x,|\nabla u^{n}|)dx$$
$$\hspace{1cm}\leq 2\tau\int_{\Omega}|f_{n}(x)||b(u^{n})|dx+2\tau b_{1}\int_{\Omega}[\varphi(x, \frac{u^{n}}{\lambda})+\varphi(x,|\nabla u^{n}|)]dx$$
$$\leq 2\tau\int_{\Omega}|f_{n}(x)||b(u^{n})|dx+4\tau b_{1}\nu_{0}\int_{\Omega}\varphi(x,|\nabla u^{n}|)dx$$
Choosing  $\nu_{0}$ such that $\nu>\frac{2b_{1}}{b_{0}}\nu_{0}=4\nu_{0}$, and summation  for all $ (n=1,2,...,N)$ implies
\[\|b(u^{n})\|_{2,\Omega}^{2}+\sum_{j=1}^{n}\|b(u^{j})-b(u^{j-1})\|_{2,\Omega}^{2}+2\tau(\nu b_{0}-2b_{1}\nu_{0})\sum_{j=1}^{n}\int_{\Omega}\varphi(x,|\nabla u^{j}|)dx\]
\[\leq 2\tau\sum_{j=1}^{n}\int_{\Omega}|f_{n}(x)||b(u^{n})|dx+\|b(u^{0})\|_{2,\Omega}^{2}.\]
Taking $\|b(u^{n})\|_{2,\Omega}=\max_{j=1,...,N}\|b(u^{j})\|_{2,\Omega}$ and
 $\tau\sum_{j=1}^{n}\|f_{n}\|^{2}_{L^{1}([0,T];L^{2}(\Omega))}\leq \|f\|^{2}_{L^{1}([0,T];L^{2}(\Omega))}$ and using Holder inequality, we obtain
\[\|b(u^{n})\|_{2,\Omega}^{2}+\sum_{j=1}^{n}\|b(u^{j})-b(u^{j-1})\|_{2,\Omega}^{2}+2\tau(\nu b_{0}-2b_{1}\nu_{0})\sum_{j=1}^{n}\int_{\Omega}\varphi(x,|\nabla u^{j}|)dx\]
\[\leq \frac{1}{\epsilon^{2}}\|f\|^{2}_{L^{1}([0,T];L^{2}(\Omega))}+\epsilon^{2} \|b(u^{n})\|_{2,\Omega}^{2}+\|b(u^{0})\|_{2,\Omega}^{2},\]
for any $0<\epsilon<1$.
From where we have \eqref{eqUnf-bund-Sol}.\\
For the second inequality we use only  \eqref{eqb} to get $b_{0}|u^{n}-u^{n-1}|\leq |b(u^{n})-b(u^{n-1})|$ and by \eqref{eqUnf-bund-Sol}, we deduce  \eqref{eqnormUn}.
\end{proof}
\end{enumerate}
\section{Existence via convergence of approximate solutions}
Let consider a two sequences $\{m_{l}\}_{l\in \NN}$ and $\{ N_{\ell}\}_{\ell\in \NN}$ such that $m_{\ell}\rightarrow \infty$, $N_{\ell}\rightarrow \infty$ as  $\ell\rightarrow \infty$ and assume that $\tau_{\ell}\leq\tau_{0}<1$ for all $\ell\in \NN$.\\
\textbf{Construction of approximate solutions ${(u_\ell)}$:} We consider a sequence $\{u^{0}_{\ell}\}_{\ell\in \NN}$ of approximation of the initial datum value such that $u_{\ell}^{0}\in V_{m}$ and
\begin{equation}\label{equ0}
    u^{0}_{\ell}\rightarrow u_{0} \quad \mbox{in}\quad L^{2}(\Omega) \quad \mbox{as}\quad \ell\rightarrow \infty.
\end{equation}
 We construct approximate solutions on the whole time interval as
follows:
\begin{itemize}
\item For the parabolic part, from  the discrete solution $\{u^{n}\}_{n=1}^{N}$ corresponding to the time step size $\tau_{\ell}=\frac{T}{N_{\ell}}$.\\
 Let $\widehat{u}_{l}$ denote the linear spline interpolating $(t_{0}=0,b(u_{\ell}^{0})),(t_{1},b(u_{\ell}^{1})),...,(t_{N_{\ell}},b(u_{\ell}^{N_{\ell}}))$, i.e.
$$\widehat{u}_{\ell}(t)=b(u_{\ell}^{n-1})+\frac{b(u_{\ell}^{n})-b(u_{\ell}^{n-1})}{\tau_{\ell}}(t-t_{n-1})\quad \mbox{for}\quad t\in [t_{n-1},t_{n}],\, (n=1,2,...,N_{\ell}).$$
\item For the elliptic part and source data,  let $u_{\ell}$ denote the piecewise constant function such that
$$u_{\ell}(.,t)=u_{\ell}^{n} \quad \mbox{if}\quad t\in (t_{n-1},t_{n}],\quad (n=1,2,...,N_{\ell}), \quad u_{\ell}(.,0)=u_{\ell}^{1}.$$
And we use the piecewise constant in time approximation $f_{\ell}$ defined by
$$f_{\ell}(.,t)=f(.,t_{n}) \quad \mbox{if}\quad t\in (t_{n-1},t_{n}],\, (n=1,2,...,N_{\ell}), \quad f_{\ell}(.,0)=f(.,t_{1}).$$
\end{itemize}
\subsection{Convergence of the numerical solution $(u_\ell)$} 
The first result of this paper can be summarized by the following Lemma.
\begin{lemma}\label{Lemm-Cvrgn-Approxim-Sol}
Let $u_{0}\in L^{2}(\Omega)$ and $f\in L^{1}(0,T;L^{2}(\Omega))$. Consider the numerical solution of \eqref{eq00000} by the scheme \eqref{eqNumMethod} on a sequence of finite dimensional subspaces such that \eqref{restriOperat-Rm} is satisfied, and time step sizes which tend to zero and are bounded away from one. For the approximation of the initial value, assume \eqref{equ0}.\\
Then, there is a subsequence, denoted by $\ell'$, and element $u\in L^{\infty}(0,T;L^{2}(\Omega))$ with $\nabla u\in (\mathcal{L}_{\varphi}(Q))^{d}$ and $\gamma_{0}u(.,t)=0$ for almost all $t\in(0,T)$, $z\in L^{2}(\Omega)$, $\alpha\in (\mathcal{L}_{\overline{\varphi}}(Q))^{d}$,  such that, as $\ell\rightarrow \infty$,
\begin{equation}\label{eqL1}
  b(u_{\ell})-\widehat{u}_{\ell}\rightarrow 0 \ \mbox{in}\ L^{2}(Q);\  b(u_{\ell'}), \widehat{u}_{\ell'}\rightharpoonup b(u) \ \mbox{in}\ L^{\infty}(0,T;L^{2}(\Omega)),
\end{equation}
  \begin{equation}\label{eqL2}
  \widehat{u}_{\ell'}(.,T)=b(u_{\ell'}(.,T))\rightharpoonup b(z) \ \mbox{in}\quad L^{2}(\Omega);\ \nabla u_{\ell'}\rightharpoonup \nabla u \ \mbox{in}\  (L_{\varphi}(Q))^{d},
 \end{equation}
\begin{equation}\label{eqL3}
K( u_{\ell'})\rightharpoonup K(u)
\ \mbox{ weakly-* in}\ (L_{\overline{\varphi}}(Q))^{d};\ a(.,\nabla u_{\ell'})\rightharpoonup \alpha \ \mbox{ weakly-* in}\ (L_{\overline{\varphi}}(Q))^{d}.
\end{equation}
\end{lemma}
\textbf{Proof of Lemma \ref{Lemm-Cvrgn-Approxim-Sol}}
\begin{enumerate}
\item  By \eqref{equ0}, the sequence $\{u_{\ell}^{0}\}$ is bounded in $L^{2}(\Omega)$. 
Using the definition of $u_{\ell}$ and $\widehat{u_\ell}$ we obtain
$$\|b(u_{\ell})-\widehat{u}_{\ell}\|^{2}_{2,Q}=\frac{\tau_\ell}{3}\sum_{n=1}^{N_{\ell}}\|b(u^{n})-b(u^{n-1})\|^{2}_{2,\Omega}$$
and by \eqref{eqUnf-bund-Sol}, we have $\lim_{l\rightarrow\infty}\|b(u_{\ell})-\widehat{u}_{\ell}\|^{2}_{2,Q}=0$.\\
Also the definition of the approximate solutions allows us to get
$$\|b(u_{\ell})\|_{L^{\infty}(0,T;L^{2}(\Omega))}=\max_{n=1,2,...,N_{\ell}}\|b(u^{n})\|_{2,\Omega},$$
$$ \|\widehat{u}_{\ell}\|_{L^{\infty}(0,T;L^{2}(\Omega))}=\max_{n=1,2,...,N_{\ell}}\|b(u^{n})\|_{2,\Omega},$$
and the inequality \eqref{eqUnf-bund-Sol} shows the boundedness of $\{b(u_{\ell})\}$ and $\{\widehat{u}_{\ell}\}$ in $L^{\infty}(0,T;L^{2}(\Omega))$.
We thus have weak* convergence of a subsequence in $L^{\infty}(0,T;L^{2}(\Omega))$. \\
Since the difference of both the sequences, tends to zero in $L^{2}(Q)$, their limits  must coincide and denoted $\varpi$. The condition \eqref{eqb} allows us to write $\varpi$ in the form $\varpi=p(u)$. That is  $b(u_{\ell'}),\widehat{u}_{\ell'}\rightharpoonup b(u) \, \ \mbox{in}\, \ L^{\infty}(0,T;L^{2}(\Omega)).$
\item  Since $\|\widehat{u}_{\ell}(.,T)\|_{2,\Omega}=\|b(u_{\ell})(.,T)\|_{2,\Omega}=\|b(u^{N_{\ell}})\|_{2,\Omega}$, the a priori estimate in \eqref{eqUnf-bund-Sol} proves the weak convergence of a subsequence of $\{\widehat{u}_{\ell}(.,T)\}$ in $L^{2}(\Omega)$ and also its limit can be written as $b(z)$ where $z\in L^{2}(\Omega)$.\\
likewise by definition \eqref{equ0} we have
$$\int_{Q}\varphi(x,|\nabla u_{\ell}|)dxdt=\tau_\ell\sum_{n=1}^{N_{\ell}}\int_{\Omega}\varphi(x,|\nabla u^{n}|)dx$$
 is uniformly bounded (see \eqref{eqUnf-bund-Sol})). However, from the boundedness of the modular boundedness of the Luxemburg
norm follows. Therefore, $\{\nabla u_{\ell}\}\subset (\mathcal{L}_{\varphi}(Q))^d\subseteq (L_{\varphi}(Q))^d$ is bounded with respect to $\|\|_{\varphi,Q}$. Since $(L_{\varphi}(Q))^{*}=E_{\overline{\varphi}}(Q)$ is  separable Banach space, we obtain weak* convergence of a subsequence in $(L_{\varphi}(Q))^d$ to an element $\xi\in (L_{\varphi}(Q))^d$ such that $\nabla u_{\ell'}\rightharpoonup^{*} \xi$.\\
   On the other hand we know that $\mathcal{C}_{c}^{\infty}(\Omega)\otimes \mathcal{C}_{c}^{\infty}(0;T)\subset E_{\overline{\varphi}}(Q)$, then for all $g\in \mathcal{C}_{c}^{\infty}(\Omega)$, $\psi\in \mathcal{C}_{c}^{\infty}(0;T)$ we get
$$\int_{Q}\xi g\psi dxdt=\lim_{\ell'\rightarrow\infty}\int_{Q}\nabla u_{\ell'} g\psi dxdt=\lim_{\ell'\rightarrow\infty}\int_{Q} u_{\ell'}\nabla g\psi dxdt=\int_{Q} u\nabla g\psi dxdt$$
since $u_{\ell'}\rightharpoonup u $ weakly in $ L^{\infty}(0,T;L^{2}(\Omega)),$ thus $\xi=\nabla u$ and in view of Lemma \ref{Lembourness}, we  get $\nabla u \in (\mathcal{L}_{\varphi}(\Omega))^d$.\\
Since the trace of $u_{\ell'}$ is zero and,
$$\int_{Q}\nabla u_{\ell'} z dxdt\rightarrow \int_{Q}\nabla u z dxdt, \quad \int_{Q}u_{\ell'} \nabla  z dxdt\rightarrow \int_{Q}u \nabla  z dxdt$$
for all $z\in  L^{\infty}(0,T;W^{1,q}(\Omega))$ with $q\geq2$  as $\ell\rightarrow\infty$, also the limit $u$ must have vanishing trace for all $t\in(0,T)$.\\
\item
Finally in view of \eqref{eqUnf-bund-Sol} and \eqref{eqa1} we have

$$\int_{Q}\overline{\varphi}(x,|a(x,|\nabla u_{\ell}|)|)dxdt\leq\int_{Q}\varphi(x,|\nabla u_{\ell}|)dxdt=\tau_\ell\sum_{n=1}^{N_{\ell}}\int_{\Omega}\varphi(x,|\nabla u^{n}_{\ell}|)dx$$
and by \eqref{eqPhi1} and \eqref{eq100} we have
$$\int_{Q}\overline{\varphi}(x,|K(u_{\ell})|)dxdt\leq \int_{Q}\varphi(x,\frac{\vert u_{\ell}\vert }{\lambda})dxdt\leq\tau_\ell C\sum_{n=1}^{N_{\ell}}\int_{\Omega}\varphi(x,|\nabla u^{n}_{\ell}|)dx$$
 are uniformly bounded. We deduce that there exists $\alpha$ and $\beta$ in $ (L_{\varphi}(Q))^{d}$ such that $a(., \nabla u_{\ell'})\rightharpoonup  \alpha$ and $K(u_{\ell'})\rightharpoonup  \beta$ weakly in $(L_{\varphi}(Q))^{d}$ for a subsequence. And by Lemma \ref{Lembourness}, $\alpha\in (\mathcal{L}_{\varphi}(Q))^{d}$ and $\beta\in (\mathcal{L}_{\varphi}(Q))^{d}$.\\
 On the other hand using \eqref{eqPhi2} we get $|K(u_{\ell})-K(u)|\leq \nu_{1}|u_{\ell}-u|$ and as we have $u_{\ell}\rightarrow u$ a.e. in $Q$ and $K$ is continuous, we obtain $\beta=K(u)$.
\end{enumerate}
\begin{remark}~\\
We omit writing $\ell'$ for the sequence from Lemma \ref{Lemm-Cvrgn-Approxim-Sol} and $x$ from $a(x,\nabla u)$ for the sake of simplicity.
\end{remark}
\begin{theorem}(Convergence of approximate solutions)\label{Cvrgn-Approxim-Sol}~\\
Let $u_{0}\in L^{2}(\Omega)$ and $f\in L^{1}(0,T;L^{2}(\Omega))$. Consider the numerical solution of \eqref{eq00000} by the scheme \eqref{eqNumMethod} on a sequence of finite dimensional subspaces such that \eqref{restriOperat-Rm} is satisfied, and time step sizes which tend to zero and are bounded away from one. For the approximation of the initial value, assume \eqref{equ0}.\\
Then there are a subsequences denoted by  $\{u_{\ell'}\}$ and $\{\widehat{u}_{\ell'}\}$ of piecewise constant in time and piecewise linear in time prolongations, respectively, of the numerical solutions converging weakly-$*$ in $L^{\infty}(0,T;L^{2}(\Omega))$ to an exact solution $u\in \mathcal{C}_{w}([0,T];L^{2}(\Omega))$ to \eqref{eq0000} and to $p(u)$ respectively. Moreover, and
 \begin{equation}\label{eqT1}
b(u_{\ell'}(.,T))=\widehat{u}_{\ell'}(.,T)\rightharpoonup b(u(.,T))\ \mbox{ weakly in}\  L^{2}(\Omega),
\end{equation}
\begin{equation}\label{eqT2}
  \nabla u_{\ell'}\rightharpoonup \nabla u \ \mbox{ weakly* in}\ (L_{\varphi}(Q))^{d} \ \mbox{and}\ \nabla u\in (\mathcal{L}_{\varphi}(Q))^{d}, \end{equation}
  \begin{equation}\label{eqT3}
  a(\nabla u_{\ell'})\rightharpoonup a(\nabla u)\ \mbox{ weakly* in}\ (L_{\overline{\varphi}}(Q))^{d}\ \mbox{and}\ a(\nabla u)\in (\mathcal{L}_{\overline{\varphi}}(Q))^{d},
  \end{equation}
\end{theorem}
\textbf{Proof of Theorem \ref{Cvrgn-Approxim-Sol}}\\
\textbf{Step 1:}~\\
Remark that by definition of $\widehat{u}_{\ell}$, the numerical scheme \eqref{eqNumMethod} can be written as
\begin{equation}\label{eqTheo1}
    \int_{\Omega}\Big[\partial_{t} \widehat{u}_{\ell}v+ (a(\nabla u_{\ell})+K(u_{\ell}))\nabla v\Big] dx= \int_{\Omega}f_{\ell}vdx \quad \mbox{for all}\quad v\in V_{m_{\ell}}
\end{equation}
 this equation holds almost everywhere in $(0,T)$ as well as in the weak sense. This implies
$$-\int_{Q}\widehat{u}_{\ell}R_{m_{l}}v\psi'dxdt+\int_{\Omega}\widehat{u}_{\ell}(.,T)R_{m_{\ell}}v\psi(T)dx-\int_{\Omega}\widehat{u}_{\ell}(.,0)R_{m_{\ell}}v\psi(0)dx$$
$$+\int_{Q}(a(\nabla u_{\ell})+K(u_{\ell}))\nabla R_{m_{\ell}} v\psi dxdt= \int_{Q}f_{\ell}R_{m_{\ell}}v\psi dxdt\ \mbox{for all} \ v\in V, \psi\in \mathcal{C}^{1}([0,T]),$$
where $\widehat{u}_{\ell}(.,T)=b(u^{N_{\ell}})$ and $\widehat{u}_{\ell}(.,0)=b(u^{0}_{\ell})$ and by \eqref{eqb} we have
\[|b(u^{N_{\ell}})-b(z)|\leq b_{1}|u^{N_{\ell}}-z|\ \mbox{and}\ |b(u^{0}_{\ell})-b(u_{0})|\leq b_{1}|u^{0}_{\ell}-u_{0}|\]
 then using Lemma \ref{Lemm-Cvrgn-Approxim-Sol} and \eqref{equ0} we obtain respectively $\widehat{u}_{\ell}(.,T)\rightarrow b(z)$ and  $\widehat{u}_{\ell}(.,0)\rightarrow b(u_{0})$ strongly in $L^{2}(\Omega)$ and also $f_{\ell}\rightarrow f$ strongly $L^{1}(0,T;L^{2}(\Omega))$.\\
Let now tends $\ell$ to $\infty$ for the others terms, we get
\begin{equation}\label{eqtheo1}
    \left\{\begin{array}{c}
      R_{m_{\ell}}v\psi'\rightarrow v\psi'\quad \mbox{in}\quad  L^{1}(0,T; L^{2}(\Omega)), \\
      \hspace{-2cm}R_{m_{\ell}}v\rightarrow v \quad \mbox{in}\quad L^{2}(\Omega),\\
     \hspace{-0.6cm} \nabla R_{m_{\ell}} v\psi\rightarrow \nabla v\psi \quad \mbox{in}\quad (E_{\varphi}(Q))^{d},\\
       R_{m_{\ell}} v\psi\rightarrow  v\psi \quad \mbox{in}\quad  L^{1}(0,T; L^{2}(\Omega)).
         \end{array}%
     \right.
\end{equation}
By appling Lemma \ref{Lemm-Cvrgn-Approxim-Sol} and \eqref{eqtheo1}, we obtain as $\ell\rightarrow\infty$\\
$\displaystyle-\int_{Q}b(u)v\psi'dxdt+\int_{\Omega}b(z)v\psi(T)dx-\int_{\Omega}b(u_{0})v\psi(0)dx$
\begin{equation}\label{eqtheo2}
+\int_{Q}(\alpha+K(u))\nabla  v\psi dxdt= \int_{Q}fv\psi dxdt \quad\mbox{for all} \quad v\in V, \psi\in \mathcal{C}^{1}([0,T]).
\end{equation}
This follows from \eqref{restriOperat-Rm} and the definition of the norm in $V$. Observe that\\ $\|\nabla R_{m_{\ell}} v\psi- \nabla v\psi\|_{\varphi,Q}\leq \|\psi\|_{\infty,[0,T]}\max(1,T)\|\nabla R_{m_{\ell}} v- \nabla v\|_{\varphi,\Omega}$ and $V\hookrightarrow W^{1,1}(\Omega)\cap L^{2}(\Omega)$.\\
Relation \eqref{eqtheo2} implies, by density arguments,
$$-\int_{Q}b(u)\partial_{t} wdxdt+\int_{\Omega}b(z)w(.,T)dx-\int_{\Omega}b(u_{0})w(.,0)dx$$
\begin{equation}\label{eqtheo3}
+\int_{Q}(\alpha+K(u))\nabla  w dx dt= \int_{Q}fw dx dt \ \mbox{for all} \ w\in \mathcal{W}.
\end{equation}
 Now, remark that the tensor product $V\otimes \mathcal{C}^{1}([0,T])\subset \mathcal{W}$, which shows that \eqref{eqtheo2} is a particular case of \eqref{eqtheo3}. The function $w_{\epsilon}$ that exists in view of Lemma \ref{LemWWWW} for any $w\in \mathcal{W}$ can be approximated, with respect to the strong convergence in $\mathcal{C}^{1}(\overline{Q})$, by a polynomial vanishing at $\partial\Omega\times [0,T]$, which possesses a tensor structure and thus belongs to $V\otimes \mathcal{C}^{1}([0,T])$. For any $u\in L^{\infty}(0,T; L^{2}(\Omega))$, $z, u_{0}\in L^{2}(\Omega)$, $\alpha\in (\mathcal{L}_{\overline{\varphi}}(Q))^{d}$, $f\in \mathcal{C}(0,T; L^{2}(\Omega))$, any $\epsilon>0$ $w\in \mathcal{W}$, there is hence (recalling also the continuous embedding of $W^{1,1}(0,T; L^{2}(\Omega))$ into $\mathcal{C}(0,T; L^{2}(\Omega))$ an element  $w_{\epsilon}\in V\otimes \mathcal{C}^{1}([0,T])$ such that
$$\displaystyle|\int_{Q}b(u)\partial_{t}(w_{\epsilon}-w)dxdt|+|\int_{\Omega}b(z)(w_{\epsilon}(.,T)-w(.,T))dx|$$
$$\displaystyle+|\int_{\Omega}b(u_{0})(w_{\epsilon}(.,0)-w(.,0))dx|+|\int_{Q}(\alpha+K(u))\nabla(w_{\epsilon}-w)dxdt|$$
$$\displaystyle+|\int_{Q}f(w_{\epsilon}-w)dxdt|<\epsilon,$$
On the other hand, since $u\in L^{\infty}(0,T; L^{2}(\Omega))$, with $\nabla u\in (\mathcal{L}_{\varphi}(Q))^{d}\subset L^{1}(0,T; L_{\varphi}(\Omega))$,\\ $\alpha\in (\mathcal{L}_{\overline{\varphi}}(Q))^{d}\subset L^{1}(0,T; L_{\overline{\varphi}}(\Omega))$, and $f\in L^{1}(0,T; L^{2}(\Omega))$, we see that for any $v\in V$ the functions
\[t\mapsto \int_{\Omega}b(u(.,t))vdx,\ t\mapsto \int_{\Omega}\alpha(.,t)\nabla vdx,\ t\mapsto \int_{\Omega}K(u)(.,t)\nabla vdx,\ t\mapsto \int_{\Omega}f(.,t)vdx,\]
 are  in $L^{1}(0,T)$ and with \eqref{eqtheo2}, we get
 \begin{equation}\label{eqTheo3}
    \frac{d}{dt}\int_{\Omega}b(u(.,t))vdx=\int_{\Omega}(f(.,t)vdx-\int_{\Omega}(\alpha(.,t)+K(u)(.,t))\nabla v)dx
 \end{equation}
holds true in the weak sense.\\
 Thus, the function $\displaystyle t\mapsto \int_{\Omega}b(u(.,t))vdx$  is absolutely continuous and since $V$ is dense in $L^{2}(\Omega)$ with respect to the strong convergence in $L^{2}(\Omega)$, we obtain $b(u)\in \mathcal{C}_{w}(0,T; L^{2}(\Omega))$.\\
\textbf{Step 2: Initial and final values}
\begin{itemize}
  \item  We  now prove $b(u(.,0))=b(u_{0})\in L^{2}(\Omega)$.\\
   For any $v\in V$, we have with \eqref{eqTheo1}
   \begin{align*}
    \displaystyle \int_{\Omega}\widehat{u}_{l}^{{0}}R_{m_{\ell}}vdx &=\Big[\int_{\Omega}\widehat{u}_{\ell}(.,t)R_{m_{\ell}}vdx\frac{t-T}{T}\Big]_{t=0}^{T}\\
  \displaystyle & =  \int_{0}^{T}\Big(\int_{\Omega}\partial_{t}\widehat{u}_{\ell}R_{m_{\ell}}vdx\frac{t-T}{T}+ \int_{\Omega}\widehat{u}_{\ell}R_{m_{\ell}}vdx\frac{1}{T}\Big)dt\\
 \displaystyle& = \int_{0}^{T}\Big([\int_{\Omega}f_{\ell}R_{m_{\ell}}vdx-\int_{\Omega}(a(\nabla u_{\ell})+K( u_{\ell}))\nabla R_{m_{\ell}}vdx]\frac{t-T}{T}\\
\displaystyle & + \int_{\Omega}\widehat{u}_{\ell}R_{m_{\ell}}vdx\frac{1}{T}\Big)dt.
   \end{align*}

Pass to the limit, we  obtain with integration by parts and using \eqref{eqTheo3}
$$\begin{array}{cc}
  \displaystyle \int_{\Omega}b(u_{0})v dx = \int_{0}^{T}\Big([\int_{\Omega}fvdx-\int_{\Omega}(\alpha+K(u))\nabla vdx]\frac{t-T}{T}+ \int_{\Omega}p(u)vdx\frac{1}{T}\Big)dt\\
\displaystyle=\Big[\int_{\Omega}b(u)v dx\frac{t-T}{T}\Big]_{t=0}^{T}=\int_{\Omega}b(u(.,0))v dx.
\end{array}%
$$
\item  Changing now the function $t\rightarrow\frac{t-T}{T}$ by $t\rightarrow\frac{t}{T}$ and using
 the same argumentation as above provides that the limit $p(z)$ of $\widehat{u}_{\ell}(.,T)=b(u_{\ell}(.,T))$ weakly in $L^{2}(\Omega)$, is exactly $b(u(.,T)),$
$$\widehat{u}_{\ell}(.,T)\rightarrow b(z)=b(u(.,T))\in L^{2}(\Omega)\quad \mbox{as}\quad \ell\rightarrow\infty.$$
\end{itemize}
\textbf{Step 3:}\begin{itemize}
     \item  For the convergence of the term $\displaystyle\int_{Q}K( u_{\ell})\nabla b(u_{\ell})dxdt$, we use \eqref{eqPhi2}  and Lemma \ref{Lemm-Cvrgn-Approxim-Sol}, to deduce that $K( u_{\ell})$ converges to $K( u)$ in $(L_{\overline{\varphi}}(Q))^{d}$ and $$\displaystyle\int_{Q}K( u_{\ell})\nabla b(u_{\ell})dxdt\rightarrow \int_{Q}K( u)\nabla b(u)dxdt.$$
\item  For the source term,  we know that
$$\int_{Q}f_{\ell}b(u_{\ell})dxdt\rightarrow \int_{Q}fb(u)dxdt \quad \mbox{as}\quad \ell\rightarrow\infty.$$

\item We now show that $\alpha=a(\nabla u).$ \\
To do so, we employ a variant of Minty's monotony trick.\\
Using $(a-b)a\geq \frac{1}{2}(a^{2}-b^{2})$, we find
\begin{align*}
  \displaystyle \int_{Q}\partial_{t}\widehat{u}b(u_{\ell})dxdt& =\sum_{n=1}^{N_{\ell}}\int_{\Omega}(b(u^{n})-b(u^{n-1}))b(u^{n})dx\\
  &\geq \frac{1}{2}(\|b(u^{N_{\ell}})\|_{2,\Omega}^{2}-\|b(u^{0}_{\ell})\|_{2,\Omega}^{2})\\
\displaystyle &=\frac{1}{2}(\|b(u_{\ell}(.,T))\|_{2,\Omega}^{2}-\|b(u^{0}_{\ell})\|_{2,\Omega}^{2}),
\end{align*}
which implies, because of the weak lower semi-continuity of the norm, the weak convergence of $u_{\ell}(.,T)$ to $z=u(T)$ in $L^{2}(\Omega)$ and the strong convergence \eqref{equ0},
\begin{equation}\label{eqTheo4}
    \frac{1}{2}(\|b(u(.,T))\|_{2,\Omega}^{2}-\|b(u^{0})\|_{2,\Omega}^{2})\leq \liminf_{\ell\rightarrow\infty}\int_{Q}\partial_{t}\widehat{u}b(u_{\ell})dxdt.
\end{equation}
For all $\eta\in (L^{\infty}(Q))^{d}$  with \eqref{eqb} and \eqref{eqa2},
  $$\int_{Q}a(\nabla u_{\ell})\nabla b(u_{\ell})dxdt\geq b_{0}\int_{Q}a(\nabla u_{\ell})\eta dxdt+b_{0}\int_{Q}a(\eta)(\nabla u_{\ell}-\eta)dxdt.$$
 Remark that $a(\eta)\in (E_{\overline{\varphi}}(Q))^{d}$ since $\eta\in (L^{\infty}(Q))^{d}$ and $a$ is continuous. In the limit, we thus obtain (see again Lemma \ref{Lemm-Cvrgn-Approxim-Sol})
\begin{equation}\label{eqTheo5}
   \liminf_{\ell\rightarrow\infty}\int_{Q}a(\nabla u_{\ell})\nabla u_{\ell}dxdt\geq b_{0}\int_{Q}\alpha\eta dxdt+b_{0}\int_{Q}a(\eta)(\nabla u-\eta)dxdt.
\end{equation}
Now taking $v=u_{\ell}(.,t)\in V_{m_{\ell}}$ in \eqref{eqtheo2} and using \eqref{eqTheo4} and \eqref{eqTheo5}, we obtain
$$\frac{1}{2}(\|b(u(T))\|_{2,\Omega}-\|b(u^{0})\|_{2,\Omega})+b_{0}\int_{Q}\alpha\eta dxdt+b_{0}\int_{Q}a(\eta)(\nabla u_{\ell}-\eta)dxdt$$
\begin{equation}\label{eqTheo6}
+\int_{Q}K( u)\nabla b(u)dxdt\leq \int_{Q}fb(u)dxdt
\end{equation}
We therefore consider the centered Steklov average of $u$, given by
$$(S_{h}u)(.,t)=\frac{1}{2h}\int_{t-h}^{t+h}b(u(.,s))ds,\quad t\in [0,T],$$
where $h>0$ and where $u$ is extended by zero outside $[0,T]$. The properties of $b$ and $u$ imply that $S_{h}u\in \mathcal{W}$. It is known that
$$\lim_{h\rightarrow 0}\int_{Q}fS_{h}udxdt=\int_{Q}fb(u)dxdt.$$
On the other hand, we find with \eqref{eqtheo3}\\
\begin{align*}
\displaystyle\int_{Q}fS_{h}udxdt=&-\int_{Q}b(u)\partial_{t}S_{h}udxdt+\int_{\Omega}b(u(.,T))S_{h}u(.,T)dx\\
&\displaystyle-\int_{\Omega}b(u(.,0))S_{h}u(.,0)dx+\int_{Q}(\alpha+K(u))\nabla S_{h}udxdt,
\end{align*}

where $\partial_{t}S_{h}u(.,t)=\frac{b(u(.,t+h))-b(u(.,t-h))}{2h},$ and thus\\
$\displaystyle\int_{Q}b(u)\partial_{t}S_{h}udxdt=\frac{1}{2h}\int_{0}^{T}\int_{\Omega}b(u(.,t))(b(u(.,t+h))-b(u(.,t-h)))dxdt$
\begin{equation}\label{eqTheo7}
=\frac{1}{2h}\int_{0}^{T-h}\int_{\Omega}b(u(.,t))b(u(.,t+h))dxdt-\frac{1}{2h}\int_{h}^{T}\int_{\Omega}b(u(.,t))b(u(.,t-h))dxdt=0
\end{equation}
Moreover, we have
\[\int_{\Omega}b(u(.,T))S_{h}u(.,T)dx=\frac{1}{2h}\int_{T-h}^{T}\int_{\Omega}b(u(.,T))b(u(.,s))dxds\]
\begin{equation}\label{eqTheo8}
    \rightarrow \frac{1}{2}\int_{\Omega}b(u(.,T))^{2}dx=\frac{1}{2}\|b(u(.,T))\|^{2}_{2,\Omega}\quad \mbox{as}\quad h\rightarrow 0.
\end{equation}
Recall here that $u\in\mathcal{C}_{w}([0,T];L^{2}(\Omega))$ and thus $s=T$ is a Lebesgue's point of the mapping $\displaystyle[0,T]\ni s\mapsto \int_{\Omega}b(u(.,T))b(u(.,s))dx$.\\
Similarly, we have
$$\int_{\Omega}b(u(.,0))S_{h}u(.,0)dx\rightarrow \frac{1}{2}\|b(u_{0})\|^{2}_{2,\Omega}\quad \mbox{as}\quad h\rightarrow 0.$$

Finally, we observe that\\
\begin{equation}\label{eqTheo9}
    \begin{array}{ll}
      \displaystyle\int_{Q}\alpha\nabla S_{h}udxdt-\int_{Q}\alpha\nabla b(u)dxdt\\
\displaystyle=\frac{1}{2}\int_{0}^{T}\int_{t-h}^{t+h}\int_{\Omega}\alpha(.,t)\nabla(b(u(.,s))-b(u(.,t)))dxdsdt\\
 \displaystyle   =\frac{1}{2}\int_{-1}^{1}\int_{0}^{T}\int_{\Omega}\alpha(.,t)\nabla(b(u(.,t+rh))-b(u(.,t)))dxdrdt \rightarrow 0 \quad \mbox{as}\quad h\rightarrow 0.
    \end{array}
\end{equation}
Similarly, also  we have
\begin{equation}\label{eqTheo10}
\int_{Q}K(u)\nabla S_{h}udxdt-\int_{Q}K(u)\nabla b(u)dxdt\rightarrow 0 \quad \mbox{as}\quad h\rightarrow 0.
\end{equation}

Since the translation of a function in the Musielak space $L_{\varphi}(Q)$ is continuous with respect to the weak convergence in $E_{\varphi}(Q)$ (see \cite{Gossez1974}).\\
Finally we obtain as $h\rightarrow 0$
\begin{equation}\label{eqTheo11}
\begin{array}{ll}
      \displaystyle \frac{1}{2}(\|b(u(T))\|_{2,\Omega}-\|b(u^{0})\|_{2,\Omega})+b_{0}\int_{Q}\alpha \nabla b(u)dxdt \\
     \displaystyle +\int_{Q}K( u)\nabla b(u)dxdt= \int_{Q}fb(u)dxdt,
 \end{array}
\end{equation}
and from \eqref{eqTheo6}, \eqref{eqTheo11} and \eqref{eqb}, we get for all $\eta\in (L^{\infty}(Q))^{d}$
$$0\leq \int_{Q}[b_{0}a(\eta)-(b_{1}-b_{0})\alpha](\eta-\nabla u)dxdt=b_{0}\int_{Q}(a(\eta)-\alpha)(\eta-\nabla u)dxdt.$$
Following the modification of Minty's trick in  \cite{Gwiazda2010Math.Methods.Appl}, we set $Q_{k}=\{(x,t):|\nabla u(x,t)|>k\}$ for any $k\in\NN$. For arbitrary $i,j\in\NN$ with $j<i$, arbitrary $\lambda>0$, and arbitrary $\zeta\in (L^{\infty}(Q))^{d}$, we take\\
$$\eta=\nabla u \chi_{Q\backslash Q_{i}}+\lambda\zeta \chi_{Q\backslash Q_{j}}=\left\{\begin{array}{ll}
                                                            0  \quad \mbox{in}\quad  Q_{i},\\
                                                            \nabla u \quad \mbox{in}\quad Q_{j}\backslash Q_{i},\\
                                                            \nabla u+ \lambda\zeta \quad \mbox{in}\quad Q\backslash Q_{j}.
                                                          \end{array}%
                                                          \right.
$$
Thus
$$0\leq -\int_{Q_{i}}(a(0)-\alpha)\nabla udxdt+\lambda\int_{Q\backslash Q_{j}}(a(\nabla u +\lambda\zeta)-\alpha)\zeta dxdt.$$
Since $(a(0)-\alpha)\nabla u\in L^{1}(Q)$, we have
$$\int_{Q_{i}}(a(0)-\alpha)\nabla udxdt\rightarrow 0 \quad\mbox{as}\quad i\rightarrow\infty,$$
 then
 $$0\leq \lambda\int_{Q\backslash Q_{j}}(a(\nabla u +\lambda\zeta)-\alpha)\zeta dxdt.$$

On the other hand since $a$ is  monotone,  for $\lambda\in[0,1]$ it yields
$$a(\nabla u +\lambda\zeta)\zeta\leq a(\nabla u +\zeta)\zeta\in L^{1}(Q),$$
then by Dominated Convergence Theorem, we get
$$\int_{Q\backslash Q_{j}}(a(\nabla u +\lambda\zeta)-\alpha)\zeta dxdt\rightarrow \int_{Q\backslash Q_{j}}(a(\nabla u)-\alpha)\zeta dxdt \quad \mbox{as}\quad \lambda\rightarrow 0,$$
and thus
$$0\leq \int_{Q\backslash Q_{j}}(a(\nabla u)-\alpha)\zeta dxdt
\quad \mbox{for any}\ j\in\NN\quad \mbox{ and any}\ \xi\in (L^{\infty}(Q))^{d}.$$
The choice
$\left\{\begin{array}{ll}
\zeta=-\frac{a(\nabla u)-\alpha}{|a(\nabla u)-\alpha|}  \quad \mbox{if}\quad  a(\nabla u)\neq \alpha,\\
\zeta=0, \quad \mbox{otherwise}.
 \end{array}%
   \right.$
\\
allows us to get
$$\int_{Q\backslash Q_{j}}|a(\nabla u)-\alpha| dxdt\leq 0,$$
and thus $\alpha=a(\nabla u)$ almost everywhere in $Q\backslash Q_{j}$, since $j$ was arbitrary, this proves the equality almost everywhere in $Q$, which complete the proof.
\end{itemize}

\textbf{Step 4: Uniqueness}~\\
Let $u$ and $v$ be two  solutions to the problem with the same data $(u_{0},f)$. From the proof above, we have
$$\left\lbrace \begin{array}{ll}
 \displaystyle \int_{Q}(b(u)-b(v))\partial_{t}wdxdt+\int_{\Omega}(b(u(.,T))-b(v(.,T)))w(.,T)dx\\
\displaystyle +\int_{Q}(a(\nabla u)-a(\nabla v))\nabla wdxdt +\int_{Q}(K( u)-K( v))\nabla w dxdt=0\, \ \mbox{for all} \, \  w\in \mathcal{W}.
\end{array}%
\right.$$
Thus
\begin{equation}\label{eqUni1}
\begin{array}{ll}
\displaystyle \int_{Q}\partial_{t}(b(u)-b(v))wdxdt\leq & \displaystyle -\int_{Q}(a(\nabla u)-a(\nabla v))\nabla wdxdt\\
& \displaystyle -\int_{Q}(K( u)-K( v))\nabla w dxdt.
\end{array}
\end{equation}
For all $\overline{t}\in[0;T]$, taking $w=\frac{1}{k}T_{k}(b(u)-b(v))w_{\epsilon,\overline{t}}$ where
$$w_{\epsilon,\overline{t}}(t)=\left  \{\begin{array}{ll}
  1 \quad\mbox{if}\quad 0\leq t\leq \overline{t}-\epsilon,\\
  \frac{\overline{t}-t}{\epsilon}\quad\mbox{if}\quad \overline{t}-\epsilon <t\leq \overline{t},\\
  0 \quad\mbox{otherwise}.
\end{array}\right.$$
In the same way as the above
$$\lim_{k\rightarrow 0}\int_{Q}(K( u)-K( v))\nabla w dxdt=0,$$
and by the monotonicity of $a$ we obtain
$$\int_{Q}\partial_{t}(b(u)-b(v))sgn(b(u)-b(v))w_{\epsilon,\overline{t}}dxdt=\lim_{k\rightarrow 0}\int_{Q}\partial_{t}(b(u)-b(v))wdxdt\leq 0.$$
On the other hand, we have
\begin{align*}
\int_{Q}\partial_{t}(b(u)-b(v))& sgn(b(u)-b(v))w_{\epsilon,\overline{t}}dxdt\\
&=\int_{0}^{\overline{t}-\epsilon}\int_{\Omega}\partial_{t}(b(u)-b(v))sgn(b(u)-b(v))dxdt\\
&+\int_{\overline{t}-\epsilon}^{\overline{t}} \frac{\overline{t}-t}{\epsilon}\int_{\Omega}\partial_{t}(b(u)-b(v))sgn(b(u)-b(v))dxdt.
\end{align*}
Employing Dominated Convergence Theorem, as $\epsilon\rightarrow 0$ we get
$$0\geq\int_{0}^{\overline{t}}\int_{\Omega}\partial_{t}(b(u)-b(v))sgn(b(u)-b(v))dxdt=\int_{0}^{\overline{t}}\frac{d}{dt}\|b(u)-b(v)\|_{L^{1}(\Omega)}dt,$$
and thus
$$\|b(u(.,\overline{t}))-b(v(.,\overline{t}))\|_{L^{1}(\Omega)}=0\quad\mbox{for all }\quad \overline{t}\in (0;T],$$
and thus the uniqueness.
\appendix
\section*{Appendix}
\textbf{Error estimate for the temporal semi-discretization}:\\
We just give the error estimate in the temporal semi-discretization case
since in the complete discretization is far from being easy to make. Let
$e_n = u (.,t_n)-u_n$ be the error between the exact solution and
numerical solution.
\begin{theorem}\label{TheoError}
    Let $u_{0},u^{0}\in L^{2}(\Omega)$, $f\in L^{1}(0,T;L^{2}(\Omega))$, $u,\partial_{t}u, \partial_{tt}^{2}u\in L^{1}(0,T;L^{2}(\Omega))$ with  $u(.,t)\in \mathcal{V}=\{v\in L^{2}(\Omega):\nabla v\in (\mathcal{L}_{\varphi}(\Omega))^{d}, \gamma_{0}u=0\}$, $a(\nabla u(.,t)), K(u(.,t))\in (\mathcal{L}_{\overline{\varphi}}(\Omega))^{d}$ for all $t\in[0,T]$. Let $u^{n}\in \mathcal{V}$ with $a(\nabla u^{n}), K(u^{n})\in  (\mathcal{L}_{\overline{\varphi}}(\Omega))^{d}$ be an approximation of $u(.,t_{n})$ such that $n=1,2,...,N,$
    $$  \int_{\Omega}\Big[\frac{b(u^{n})-b(u^{n-1})}{\tau}v+a(\nabla u^{n})\nabla v+K(u^{n})\nabla v\Big]dx=\int_{\Omega}f(.,t_{n})vdx\quad \mbox{for all}\quad v\in \mathcal{V}.$$
    Then for $n=1,2,...,N,$
        \begin{equation}\label{eqe0}
    \|e^{^{n}}\|_{L^{1}(\Omega)}\leq C(\|u_{0}-u^{0}\|_{L^{1}(\Omega)}+\tau\|\partial_{tt}^{2}u\|_{L^{1}(0,T;L^{2}(\Omega))}+\|\overline{f}-f\|_{L^{1}(\Omega)})
    \end{equation}
    where $\overline{f}$ denotes the piecewise constant in time interpolation of $f$ with respect to $(t_n)_{n=1}^{N}$.
    \end{theorem}
\begin{proof}
Let $R^{n}=b(u(.,t_{n}))-b(u^{n})$ then
$$ \int_{\Omega}\frac{R^{n}-R^{n-1}}{\tau}vdx=\int_{\Omega}\frac{b(u(.,t_{n}))-b(u(.,t_{n-1}))}{\tau}vdx-\int_{\Omega}\frac{b(u^{n})-b(u^{n-1})}{\tau}vdx.$$
By integrate par parts and\eqref{eqNumMethod} we have $$\int_{\Omega}\frac{R^{n}-R^{n-1}}{\tau}vdx=\int_{\Omega}\partial_{t}b(u(.,t_{n}))vdx-\frac{1}{\tau}\int_{t_{n-1}}^{t_{n}}\int_{\Omega}(t-t_{n-1})\partial_{tt}b(u(.,t))vdxdt$$
$$\hspace{-2cm}-\int_{\Omega}\frac{b(u^{n})-b(u^{n-1})}{\tau}vdx$$
$$\hspace{2cm}=\int_{\Omega}(f(.,t_{n})-f_{n})vdx -\int_{\Omega}(a(\nabla u(.,t_{n}))-a(\nabla u^{n}))vdx$$
\begin{equation}\label{eqR1}
-\int_{\Omega}(K( u(.,t_{n}))-K( u^{n}))vdx-\frac{1}{\tau}\int_{t_{n-1}}^{t_{n}}\int_{\Omega}(t-t_{n-1})\partial_{tt}b(u(.,t))vdxdt
\end{equation}
Taking $v_{k}=\frac{1}{k}T_{k}(R^{n})$ and using \eqref{eqa2}, we have
$$|\int_{\Omega}(f(.,t_{n})-f_{n})vdx|\leq\|f(.,t_{n})-f_{n}\|_{L^{1}(\Omega)} $$
$$-\int_{\Omega}(a(\nabla u(.,t_{n}))-a(\nabla u^{n}))\nabla v dx\leq -p_{0}\int_{\{|R^{n}|\leq k\}}(a(\nabla u(.,t_{n}))-a(\nabla u^{n}))$$
$$\hspace{3cm}\times\nabla(u(.,t_{n})-u^{n})dx\leq0$$
$$\frac{1}{\tau}\int_{t_{n-1}}^{t_{n}}\int_{\Omega}(t-t_{n-1})\partial^{2}_{tt}b(u(.,t))vdxdt\leq \frac{1}{\tau}\int_{t_{n-1}}^{t_{n}}\int_{\Omega}(t-t_{n-1})|\partial^{2}_{tt}b(u(.,t))|dxdt,$$
and
$$|\int_{\Omega}(K( u(.,t_{n}))-K( u^{n}))vdx|\leq \frac{p_{0}\nu_{1}}{k}\int_{\{|R^{n}|\leq k\}}|u(.,t_{n})- u^{n}||\nabla u(.,t_{n})- \nabla u^{n}|dx$$
$$\hspace{3cm}\leq \frac{p_{0}\nu_{1}}{k}\int_{\{|R^{n}|\leq k\}}|\nabla u(.,t_{n})- \nabla u^{n}|dx.$$
Since $|\nabla u(.,t_{n}))- \nabla u^{n}|\in L^{1}(\Omega)$,
 we pass to the limit as $k\rightarrow +\infty$, and we have
$$\lim_{k\rightarrow +\infty}\int_{\Omega}(K( u(.,t_{n}))-K( u^{n}))vdx=0,$$
and also
$$ \lim_{k\rightarrow +\infty}\int_{\Omega}\frac{R^{n}-R^{n-1}}{\tau}v_{k}dx=\int_{\Omega}\frac{R^{n}-R^{n-1}}{\tau}sgn(R^{n})dx.$$
Recalling \eqref{eqR1}, we obtain
$$\int_{\Omega}\frac{|R^{n}|}{\tau}dx\leq \|f(.,t_{n})-f_{n}\|_{L^{1}(\Omega)} +\frac{1}{\tau}\int_{t_{n-1}}^{t_{n}}\int_{\Omega}(t-t_{n-1})|\partial^{2}_{tt}b(u(.,t))|dxdt +\int_{\Omega}\frac{|R^{n-1}|}{\tau}dx.$$
Thus summation $n=1,...,N$ and using \eqref{eqb}
$$b_{0}\|e^{n}\|_{L^{1}(\Omega)}\leq\tau \sum_{n=1}^{N}\|f(.,t_{n})-f_{n}\|_{L^{1}(\Omega)} +2\tau\int_{0}^{T}|\partial^{2}_{tt}b(u(.,t))|_{L^{1}(\Omega)}dt +p_{1}\|e^{0}\|_{L^{1}(\Omega)}.$$
Thus, together with the estimate $\tau \sum_{n=1}^{N}\|f(.,t_{n})-f_{n}\|_{L^{1}(\Omega)}\leq c\|\overline{f}-f\|_{L^{1}(\Omega)}$,
and $L^{1}(0,T;L^{2}(\Omega))\subset L^{1}(0,T;L^{1}(\Omega))$, we deduce \eqref{eqe0}.
\end{proof}
\textbf{Comments}: The lack of regularity results for weak solutions in Museilak spaces makes the task of showing convergence results for such regular solutions very difficult.


\end{document}